\documentclass[12pt,a4paper,leqno]{amsart}
\usepackage[numbers]{natbib}
\usepackage{hyperref}

\usepackage[utf8]{inputenc}

\usepackage{amsmath}
\usepackage[T1]{fontenc}
\usepackage{amsthm}
\usepackage{amssymb}
\usepackage{csquotes}
\usepackage{enumerate}
\usepackage[dvipsnames]{xcolor}
\usepackage{lmodern}
\DeclareMathAlphabet{\mathpzc}{OT1}{pzc}{m}{it}

\renewcommand{\epsilon}{\varepsilon}

\newcommand{\bbQ}{\mathbb{Q}}
\newcommand{\bbN}{\mathbb{N}}
\newcommand{\bbR}{\mathbb{R}}

\newcommand{\calC}{\mathcal{C}}

\newcommand{\calP}{\mathcal{P}}

\newcommand{\calV}{\mathcal{V}}

\newcommand{\calA}{\mathcal{A}}
\newcommand{\calB}{\mathcal{B}}
\newcommand{\calI}{\mathcal{I}}
\newcommand{\calT}{\mathcal{T}}
\newcommand{\calZ}{\mathcal{Z}}

\renewcommand{\epsilon}{\varepsilon}
\renewcommand{\int}{\operatorname{int}}
\renewcommand{\phi}{\varphi}

\newcommand{\inte}{\operatorname{int}}

\newtheorem{thm}{Theorem}
\newtheorem{pro}[thm]{Proposition}
\newtheorem{lem}[thm]{Lemma}
\newtheorem{cor}[thm]{Corollary}
\newtheorem{fac}[thm]{Fact}

\theoremstyle{definition}
\newtheorem{rem}[thm]{Remark}

\title{On compact subsets of the reals}


\author{Wojciech Bielas}
\address{Institute of Mathematics, University of Silesia in Katowice, Bankowa~14, 40-007 Katowice}
\email{wojciech.bielas@us.edu.pl}
\email{mateusz.kula@us.edu.pl}
\email{szymon.plewik@us.edu.pl}
\author{Mateusz Kula}
\author{Szymon Plewik}

\begin{document}

\begin{abstract} Motivated by results of J. R. Kline and R. L. Moore (1919)
that a compact subset of the plane, homeomorphic to a subset of the reals, lies on the arc, we give a purely topological characterisation of compact sets of the reals. 
This allows us to reduce investigations of Cantorvals to properties of countable linear orders and to show,
applying the Mazurkiewicz--Sierpi\'nski Theorem (1920), that there exist continuum many non-homeomorphic
$L$-Cantorvals.  \end{abstract}

\maketitle
\section{Introduction}
If $X$ is a compact subset of the reals, then

	 $(\star)$. Each non-trivial component of $X$ is an arc $[a,b]$ and the set $(a,b)=[a,b]\setminus\{a,b\}$ is open.
	
\noindent J. R. Kline and R. L. Moore \cite[Theorem 1]{mk} showed that a compact subset of the plane $\bbR^2$ that satisfies condition $(\star)$, lies on an arc. 
We are impressed by  this achievement, more than 100 years old, and we have decided to  continue this.
While Kline and Moore dealt only with compact subsets of the plane, their results had been generalised,  for example  \cite{zip} and \cite{zipc}.
Here, we present  a variation  that gives  an internal characterisation of compact subsets of the reals.
Also, we have found some applications to Cantorvals.
As far as we know, the name Cantorval has been introduced by P. Mendes and F. Oliveira \cite{mo}.
The first characterization of Cantorvals was given by J. A. Guthrie and J. E. Nymann \cite{gn}.
As it is noted by S. G\l\k{a}b and J. Marchwicki \cite{mg},  a few authors pointed out the existence of some concrete Cantorvals before Guthrie and Nymann. 

By an \emph{arc} we understand a closed interval $[a,b]\subseteq\bbR$, or a homeomorphic image of a closed interval $[a,b]$, which also will be denoted by $[a,b]$.
In both cases, points $a$ and $b$ are   called \emph{endpoints} of the arc $[a,b]$.
Metric notions, which are described in \cite[pp. 248--254]{enge}, we supplement with
$$B^*(Y,\varepsilon)=\int_X Y\cup B(Y\setminus\int_XY,\varepsilon),$$
where $B(Y\setminus \int_XY,\varepsilon)$ is the $\epsilon$-ball around the set $Y\setminus\int_XY$.
If $Y$ is a compact subset of the reals $\mathbb{R}$, then any  component of $\bbR\setminus Y$ is called a $Y$-\textit{gap} and any non-trivial component of $Y$ is called a $Y$-\textit{interval}.
When the set $Y$ is fixed, we say  gap and interval instead of $Y$-\textit{gap} and $Y$-\textit{interval}, respectively.
Let us recall  the following well-known facts.
\begin{fac}
\label{fac1} If $X$ is a compact Hausdorff space, then any component $S\subseteq X$ is the intersection of the family of all clopen sets containing $S$.
\end{fac}
\begin{proof}In compact Hausdorff spaces, components coincide with quasi-components, see  \cite[Theorem 6.1.23]{enge}.
\end{proof}
\begin{fac}
\label{fac2}
 Let $X$ be a compact Hausdorff space. If a component $S$ is contained in an open set $U$, then  there exists a clopen set $V$ such that $S\subseteq V \subseteq U$.
\end{fac}
\begin{proof}
  In view of Fact 1, choose a family $\mathcal V$ of clopen sets such that  $S=\bigcap\mathcal V$.
  Choose a finite subfamily $\mathcal V^*\subseteq \mathcal V$ such that $X\setminus\bigcap \mathcal V^*\supseteq X\setminus  U$.
  Therefore $\bigcap\calV^*\subseteq U$.
\end{proof} 

Our notation is standard, in particular, a closed and open set  is called \emph{clopen}. For notions not defined here, see \cite{enge} and \cite{ss}.

\section{The Kline--Moore theorem revisited}

If a compact metric space $X$ satisfies condition $(\star)$, then there are at most countably many non-trivial components.
Indeed,  each non-trivial component of $X$ has non-empty interior, different components are disjoint and $X$ is separable.
	
Let $\calB_n$ be the family of all clopen sets $V\subseteq X$ for which there exists a component $S$ of $X$ such that
$$\textstyle S\subseteq V\subseteq B^*(S,\frac{1}{n}).$$
Clearly,  families $\calB_n$ are defined for any metric space $X$ and $\calB_{n+1}\subseteq\calB_{n}$.
We say that a family $\mathcal B$ of clopen sets of $X$ is a \textit{base for components}, whenever any component of $X$ can be obtained as the intersection of a family contained  in $\calB$.
If $X$ is a compact metric space, then the family $\calB_n$ is a base for components, for each $n>0$.
Indeed, if $S$ is a component of $X$, then $S=\bigcap \mathcal V$, where $\mathcal V$ is a family of clopen sets.
For each $V\in\mathcal V$, by Fact \ref{fac2},  there exists a clopen set $U_V$ such that
$$\textstyle S\subseteq U_V\subseteq V\cap B^*(S,\frac 1n).$$
Therefore $S=\bigcap\{U_V \colon V\in \mathcal V\}$ and all $U_V$ belong to $\calB_n$.

Let $X$ be a metric space and  $V\subseteq X$ be compact and  clopen.
Each of the following facts is  more or less obvious.
\begin{itemize}
\item If $V$ is the sum of finitely many components, then each such component is open and belongs to  $\calB_n$.
\item If $V$ contains infinitely many components and the diameter $\delta(V)$ is less than $\frac 1n$, then $V\in\calB_n$. Indeed, by compactness of $V$ one of the components contained in $V$ is not open, so there is a component $S\subseteq V$ and a point $x\in S\setminus\int_XS$, and $$\textstyle S \subseteq V \subseteq B(x,\frac 1n) \subseteq B^*(S,\frac{1}{n}).$$
\item If  $S$ is a component of $X$ such that $S \subseteq V \subseteq B^*(S,\frac{1}{n})$ and $W\subseteq X\setminus S$ is a clopen set, then $S\subseteq V\setminus W \in \calB_n$. 
\end{itemize}

\begin{pro}\label{pro1} Let $(X,d)$ be a compact metric space. If each $\calP_n$  is a cover consisting of clopen sets and each $\calP_n$ is a refinement of $\calB_n$, then $\bigcup\{\calP_n:n >0\}$ is a base for components.
\end{pro}

\begin{proof} Let $S\subseteq X$ be a component and, for each $n>0$, take   the clopen set $V_n$ such that $S\subseteq V_n\in\calP_n$.
So, there exists components $S_n$ such that
  $$\textstyle S\subseteq V_n\subseteq W_n \subseteq B^*(S_n,\frac 1n)\subseteq B(S_n,\frac 1n),$$
	where each $W_n\in\calB_n$ witnesses that $\calP_n$ is a refinement of $\calB_n$ and each $S_n$ witnesses that $W_n\in\calB_n$.
	We get $S=\bigcap\{V_n:n >0\}$.
        Indeed, suppose
        $$a\in \bigcap\{V_n\colon n>0\}\setminus S.$$
        Choosing   a clopen set $U\supseteq S$ such that $a\notin U$,       we have
        $$\epsilon=d(S,X\setminus U)>0.$$
Let points  $b_n\in S_n$ be such that $d(a,b_n)<\frac 1n$, therefore $\lim_{n\to\infty} b_n = a$.
For all but finitely many $n$, we have  $b_n\in S_n\subseteq X\setminus U$ and $B(S,{\epsilon})\cap S_n=\emptyset$, and consequently $S\cap B(S_n,\epsilon)=\emptyset$.
On the other hand, for sufficiently large $n$, we get $$\textstyle S\subseteq V_n\subseteq B(S_n,\frac{1}{n})\subseteq B(S_n,\epsilon);$$ a contradiction.
\end{proof}

Recall that a family $\calP$ of subsets of $X$ is called a \textit{partition}, whenever  $\bigcup\calP=X$ and $\calP$ consists of non-empty pairwise disjoint sets.

\begin{lem}
\label{lem6bis}
If  $X$ is a compact metric space, a set $V\subseteq X$ is clopen and $n>0$, then there exists a family $\calP$, consisting of pairwise disjoint clopen sets, that is a refinement of $\calB_n$ and $V=\bigcup \calP$.
In addition, if $X$ satisfies condition $(\star)$, then $\calP$ can be chosen such  that $\calP\subseteq \calB_n$.
\end{lem}
\begin{proof}
  Fix $n>0$.
  For each component $S\subseteq V$, using Fact \ref{fac2}, choose a clopen set $V(S)$ such that
$$\textstyle S\subseteq V(S)\subseteq B^*(S, \frac1n)\cap V.$$
Thus sets $V(S)$ constitute a cover of $V$, so there exists  a finite cover $\{V(S_1),\ldots,V(S_k)\}\subseteq \calB_n$.
If
$$W_i=V(S_i)\setminus\bigcup \{V({S_j})\colon j <i\},$$
then  sets $W_i\neq \emptyset$ form a partition of $V$ that is a refinement of $\calB_n$.

Assume that condition $(\star)$ is fulfilled.
For each  non-trivial component $S$ let the sets $V(S)$ be such that
$$\textstyle S\subseteq V(S)\subseteq B^*(S, \epsilon)\cap V,$$
where $\epsilon<\frac1{2n}$ is smaller than  half the distance between the endpoints of $S$.
Let the sets $W_i$ be defined  as above.
If $S_i\subseteq W_i$, then $W_i\in\calB_n$. Otherwise $S_i\cap W_i=\emptyset$ and $W_i$ is a union of at most two disjoint clopen sets, each one of diameter less than $\frac 1n$. Any clopen set of diameter less than $\frac 1n$ with infinitely many components belongs to $\calB_n$. But, if a clopen set has finitely many components, then these components are open and belong to $\calB_n$.
Whatever the case is, there exists a partition of $W_i$ contained in $\calB_n$.
\end{proof}

\begin{lem}
\label{lem3}
In a compact metric space $X$, there exists a sequence $(\calP_n)$ of finite partitions   of clopen sets  such that each $\calP_n$ is a refinement of $\calB_n$ and  $\calP_{n+1}$ is a refinement of $\calP_{n}$.
In addition, if $X$ satisfies condition $(\star)$, then $\calP_n$ can be chosen such that $\calP_n\subseteq \calB_n$.
\end{lem}
\begin{proof}
Using Lemma \ref{lem6bis}, define inductively partitions $\calP_1,\calP_2,\ldots$ such that $\calP_{n}$ refines $\calB_{n}$ and $\calP_{n}$ refines $\calP_{n-1}$, while $\calP_0=\{X\}$. If $X$ fulfils condition $(\star)$, then we apply the second part of Lemma \ref{lem6bis}.
\end{proof}

\begin{lem}
\label{lem8}
Let $X$ be a compact metric space and $(\calP_n)$ be a sequence of partitions as in Lemma \ref{lem3}.
If $S\subseteq X$ is a component and $S\subseteq V_n\in\calP_n$ always hold, then $S=\bigcap\{V_n: n>0\}$. 
\end{lem}
\begin{proof}
Let $S\subseteq X$ be a component and $S\subseteq V_n\in\calP_n$ for each $n$.
Then the family $\mathcal U=\{V_n: n>0\}$ consists of exactly those elements of $\bigcup\{\calP_n\colon n>0\}$ that contain $S$.
We have $\bigcap\mathcal U=S$, since $\bigcup\{\calP_n\colon n>0\}$ is a base for components.
\end{proof}

If a compact metric space $(X,d)$ and $W\subseteq X$ is a clopen set, then an arc $[a,b]\subseteq W$ is said to be a \emph{core} of $W$, whenever $[a,b]$ is a non-trivial component such that 
 $$\textstyle [a,b]\subseteq W\subseteq B^*([a,b],\frac12d(a,b))\mbox{ and }(a,b)\mbox{ is open}.$$
 Not every clopen set has a core, but if it is the case, then such a core is unique.
 Indeed, if $S_1=[a_1,b_1]$ and $S_2=[a_2,b_2]$ are different cores of $W$, then $S_2\subseteq B(a_1,\frac12d(a_1,b_1))$ or $S_2\subseteq B(b_1,\frac12d(a_1,b_1))$, which implies $d(a_2,b_2)<d(a_1,b_1)$.
The interchange of $S_1$ and $S_2$ gives $d(a_1,b_1)<d(a_2,b_2)$; a contradiction.

\begin{thm}\label{thm:9}
If a compact and metric space $(X,d)$ satisfies condition $(\star)$,
then $X$ is homeomorphic to a subset of the reals.
\end{thm}
\begin{proof}
  For each non-trivial component $[a,b]$ of $X$,  choose a homeomorphism $f_{[a,b]}\colon[a,b]\rightarrow [0,1]$ such that $f_{[a,b]}(a)=0$.
  %Let $\epsilon_I=\frac13d(a_I,b_I)$.
  Let $(\calP_n)$ be a sequence of finite partitions of $X$   such that $\calP_n\subseteq \calB_n$ and $\calP_{n}$ is a refinement of $\calP_{n-1}$.
  Without loss of generality, by modifying the construction of partitions $\calP_n$ used in Lemma \ref{lem3}, assume that if $W\in\calP_{n-1}$ has the core $[a,b]$, then for any $V\in\calP_n$ such that  $V\subseteq W$, we have $V\subseteq B(a,\frac 1n)$ or  $V\subseteq B(b,\frac 1n)$, or $[a,b]\subseteq V$.
  Define (linear) strict ordering $\sqsubset_n$, each one on the family $\calP_n$, such that the following conditions hold.
\begin{enumerate}
	\item[(1)] $V_1\sqsubset_{n} V_2$, whenever $V_1,V_2\in\calP_{n}$, $V_1\subseteq W_1\in\calP_{n-1}$, $V_2\subseteq W_2\in\calP_{n-1}$ and $W_1\sqsubset_{n-1} W_2$.
	\item[(2)] If $W\in\calP_{n-1}$ has the core $[a,b]\subseteq V^*\in\calP_n$, then for any $V\in\calP_n$, such that $V\subseteq W$, we have: if $V\subseteq B(a,\frac 1n)$, then  $V\sqsubset_n V^*$; if  $V\subseteq B(b,\frac 1n)$, then  $V^*\sqsubset_n V$.
        \end{enumerate}
        
Assume inductively that the order $\sqsubset_{n}$ is already defined. Given two sets $V_1,V_2\in\calP_{n+1}$ the order $\sqsubset_{n+1}$ between them is determined by condition (1) unless they are subsets of the same set $W\in\calP_{n}$.
In such a case,  define an arbitrarily linear order on  $\{V\in\calP_{n+1}\colon V\subseteq W\}$, maintaining condition (2).

Given orders $\sqsubset_n$ we order linearly the family $\mathcal C$ of all components of $X$ as follows. 
Let $S$, $S'$ be two different components of $X$.
  Consider sequences of clopen sets $(W_n),(W_n')$ such that $W_n,W_n'\in\calP_n$ and $\bigcap_nW_n=S$, and $\bigcap_nW_n'=S'$.
  If $n$ is the smallest number such that $W_n\neq W_n'$, then let $\sqsubset$ order  $S$ and $S'$ in the same manner as $\sqsubset_n$ orders sets $W_n$ and $W_n'$.
If $\mathcal C^*$ is the union of sets $\{(S, 0)\colon S \text{ is a trivial component}\}$ and $\{(I, x) \colon I \text{ is a non-trivial component and } x\in [0,1]\}$, then equip it  with the lexicographic order $\prec$ induced by $\sqsubset$ and the standard inequality $<$.

The space $X$ is homeomorphic to $\mathcal C^*$ with the order topology.
Indeed, for any $x\in X$, put
$$F(x)=
\begin{cases}
([a,b],f_{[a,b]}(x)),&\mbox{ if }x\in [a,b]\in\mathcal C;\\
(S,0),&\mbox{ if }S=\{x\}\in \mathcal C.
\end{cases}$$
Clearly $F\colon X\rightarrow \mathcal C^*$ is a bijection.

%Let us  check that sets $F^{-1}[(F(x),\rightarrow)]$ and $F^{-1}[(\leftarrow,F(x))]$ are always open.
Fix $x\in X$ and a sequence of sets $(W_n)$ such that $x\in W_n\in\calP_n$.
If $\{x\}$ is a trivial component or is the right endpoint of a non-trivial component $[a,b]$, then
 $$F^{-1}[(F(x),\to)]=\bigcup \{V\colon V\in\calP_n\mbox{ and } W_n\sqsubset_n V\mbox{ and }n>0\}$$
 is open, being a union of open sets.
 
 But if  $I=[a,b]\subseteq X$ is a non-trivial component and $x\in [a,b)$, then we have that the set 
 $$F^{-1}[(\leftarrow,F(x)]]=[a,x]\cup\bigcup \{W\in\calP_n\colon W\sqsubset_n W_n\mbox{ and }n>0\}$$
 is closed.
Thus, $F$ is continuous and since the domain of  $F$ is compact, hence $F$ is a homeomorphism.

In the space $\mathcal C^*$ there exist at most countably many gaps, i.e. points with  immediate successors with respect to $\prec$.
Filling these gaps with copies of the interval $(0,1)$, we obtain an arc that contains a copy of $X$.
\end{proof}

\begin{rem}\label{rem10}
In the proof of the above theorem the ordering of endpoints of non-trivial components did not matter.
\end{rem}

\section{Some remarks on Cantorvals} \label{sec:1}

Authors of  the paper \cite{mo} considered the arithmetic sum of two Cantor sets on the reals.
For some classes of such sets the arithmetic sum can be of the form which they called: M-Cantorval, L-Cantorval or R-Cantorval. Each of these notions represents a compact subset of the reals  with given properties. Also, nowadays it is common to call what they meant by M-Cantorval simply Cantorval or sometimes symmetric Cantorval, see \cite[p. 16]{nit2} or compare \cite{nit1}. 
There are a few  approaches to the definition of a Cantorval. 
For example, a \textit{Cantorval} is a nonempty compact subset $K\subseteq \mathbb R$   such that
 $K$ is the closure of its interior and 
 endpoints of any non-trivial component of $K$ are limit points
of trivial  components of $K$, see \cite[Definition 13]{nit2}.
Let us describe a specific Cantorval, compare \cite[p. 354]{bfpw}. Consider the classical  ternary Cantor set and let $G_n$ be the union of $2^{n-1}$ open intervals removed at the $n$-th step of its construction, i.e. 
 $$\textstyle G_1=\left(\frac{1}{3},\frac{2}{3}\right) \subseteq \mathbb R \mbox{  and } G_{n+1}=\frac{1}{3}G_n\cup\left(\frac{1}{3}G_n+\frac{2}{3}\right).$$ The set $$C=[0,1]\setminus\bigcup_{k\in\bbN}G_{2k}$$
 is a Cantorval and we will call it \textit{model Cantorval}. As it was shown, see   \cite[p. 343]{mo}, \cite[p. 16]{nit2}, \cite[p. 355]{bfpw}, \cite[Theorem 1.3]{prus} or \cite[Wniosek 26.]{wal}, any two Cantorvals are homeomorphic.
Topological characterizations of a Cantorval are given in \cite[Theorem 21.18.]{bfpw} and \cite[pp. 330/331]{mo}.
Another  characterization can be also extracted from the proofs of Guthrie--Nymann's classification theorem on sets of subsums of a convergent series of real numbers, see \cite[Theorem 1.(iii)]{gn}.

Following the paper \cite{mo},  a  subset $X$ of the reals is an \emph{L-Cantorval} [respectively \emph{R-Cantorval}] whenever it is nonempty, compact and any gap has an interval adjacent to its right  [left] and is accumulated on the left [right] by infinitely many intervals and gaps.
Thus, if $X$ is an L-Cantorval, then $\min X$ is the endpoint of the $X$-interval.
Directly from the definition, we get that the map $x\mapsto -x$ carries an L-Cantorval onto the R-Cantorval and vice versa.
Also, if  $[a,b]$ is an interval in an L-Cantorval, then $b$ cannot be an endpoint of any gap, but $a$ is not necessarily  an endpoint of a gap, i.e. $a$ can be accumulated by infinitely  many gaps and intervals.
It is known that there exist non-homeomorphic L-Cantorvals, see the last paragraph of Appendix in \cite{mo}.
As we show, there exist continuum many non-homeomorphic L-Cantorvals, see Theorem \ref{thm11}.

We say that an endpoint $e$ of a non-trivial component $I\subseteq X$ is  \textit{free}, whenever $X$ satisfies condition $(\star)$ and $e\in\int_X I$. A set $X\subseteq \bbR$ is called \textit{left-oriented}, whenever any $X$-interval with exactly one free endpoint has the left endpoint free.
If $(a,b)$ is an $X$-gap such that  the point $b$ is a limit point of free left endpoints of $X$-intervals, then   $(a,b)$ is called  an \emph{inappropriate} $X$-gap.
If $X$ is left-oriented, each   non-empty clopen subset of $X$ contains a non-trivial component with exactly one free endpoint, $\{b\}$ is a trivial component and $b$ is the right endpoint of an $X$-gap $(a,b)$, then $(a,b)$ is an inappropriate gap.
If $X\subseteq \bbR$,  then we say that a clopen subset $W\subseteq X$ is \textit{segment-like} with respect to $X$, if $W=X\cap I$ where $I$ is a closed interval and $\min W$ belongs to an $X$-interval.

Assume that $X\subseteq\bbR$ is a compact set and $(a,b)$ is an inappropriate  $X$-gap.  We say that a pair consisting of a function $f\colon X\rightarrow f[X]$ and a family
$\calP=\{V_k\colon k\geqslant 0\}$
of non-empty clopen sets  is a \textit{correction} of $(a,b)$ with respect to $\epsilon>0$, if the following conditions are fulfilled.
\begin{itemize}
		\item $\bigcup\{V_k\colon {k\geqslant 0}\}=X\cap (b, b+\eta)$ where $\eta<\min\{b-a,\epsilon\}$.
	\item $\max V_{k+1}<\min V_k$ and each $V_k$ is segment-like.
	\item $$f(x)=
	\begin{cases}x-\delta_k, &x\in V_k;\\
	x,&x\in X\setminus \bigcup\{V_k\colon {k>0}\},
	\end{cases}$$ where $\delta_k=\max V_k+\min V_k-2b$.
      \end{itemize}
Clearly, if $f$ is a correction of $(a,b)$ with respect to $\epsilon$, then 
	 $f$ is a homeomorphism,
	 for any $x\in X$, we have $0\leqslant x-f(x)< 2\epsilon$ and if $X$ is left-oriented, then $f[X]$ is left-oriented. Also, the $f[X]$-gap $(b, \min V_0)$ is not  inappropriate.  

\begin{thm}\label{thm11}
  Let $(X,d)$ be a compact metric space that satisfies condition $(\star)$.
If each   non-empty clopen subset of $X$ contains a non-trivial component with exactly one free endpoint, then $X$ is homeomorphic to an L-Cantorval.
\end{thm}

\begin{proof}
  By Theorem \ref{thm:9} and Remark \ref{rem10}, we can assume that $X\subseteq\bbR$ and $X$ is left-oriented.
Assume that  $(a_1,b_1),(a_2,b_2),\ldots$ is an enumeration of all inappropriate $X$-gaps.
Let  $f_1\colon X\to X_1$ and   $\calP_1^*=\{V_k\colon k\geqslant 0\} $ be a correction of the $X$-gap $(a_1,b_1)$ with respect to $\epsilon=1$. Put $\calP_1= \{f_1[V_k]\colon k>0\}\cup \{ (-\infty, a_1] \cap X, [\min V_0, +\infty)\cap X\}$.

For  $i<n$, assume inductively, that there are already defined families $\calP_i$ of clopen sets in $X_i$ and homeomorphisms $f_i\colon X_{i-1}\to X_i\subseteq\bbR$.
% such that each pair $f_i$ and $\calP_i^*$ is a correction of the $X_{i-1}$-gap   with respect to $2^{-i}$. 
Let $f_n\colon X_{n-1}\to X_n$ and a family $\calP^*=\{V_k\colon k\geqslant 0\}$ be a correction of an $X_{n-1}$-gap $(p_n,q_n)$ with respect to $2^{-n}$, where $$(p_{n},q_{n})=(f_{n-1}\circ \ldots\circ f_1(a_n),f_{n-1}\circ \ldots\circ f_1(b_n)).$$
Choose a family  $\calP_n$, which consists of:
\begin{enumerate}
\item all sets from $\calP_{n-1}\setminus\{W\}$, where $p_n\in W\in \calP_{n-1}$;
\item the set $W\cap[\min V_0,\max W]$;
\item the set $W\cap[\min W,p_n]$;
\item  sets $f_n[V_k]\subseteq (p_n,q_n)$, for $k>0$.
\end{enumerate}
Since each $f_n$ is a correction of the $X_{n-1}$-gap $(p_{n},q_{n})$ with respect to $\epsilon=2^{-n}$, then  the sequence $(f_n\circ \ldots\circ f_1)$ is uniformly convergent to the continuous surjection $h\colon X\to X_\infty$.

Suppose $x_0,x_1\in X$, $x_0\neq x_1$ and $h(x_0)=h(x_1)$.
Without loss of generality, we can assume that $f_n\circ\ldots\circ f_1( x_0)\in V\in\calP_n$ and $f_n\circ\ldots\circ f_1( x_1)\notin V$, where $n>0$.
Therefore $f_m\circ\ldots\circ f_1(x_0)\in [\min V,\max V]$ and $f_m\circ\ldots\circ f_1(x_1)\notin [\min V,\max V]$ for each $m\geqslant n$, which is in conflict with $h(x_0)=h(x_1)$.
Thus, $h$ is a continuous bijection, i.e. $h\colon X\to h[X]=X_\infty\subseteq\bbR$ is a homeomorphism.

If $J$ is an arbitrary $X$-interval, then  $f_{n+1}\circ\ldots\circ f_1|_J=f_n\circ\ldots\circ f_1|_J$ for sufficiently large $n$.
Therefore $X_\infty$ is left-oriented, since each $f_n\circ\ldots\circ f_1$ preserves this property.

Suppose that $(c,d)$ is an inappropriate $X_\infty$-gap.
If
$$d_n=f_n\circ\ldots\circ f_1(h^{-1}(d)),$$
then the sequence $(d_n)$  converges to $d$ and there exists $n$ such that $d_{n+1}<d_n$, since $h^{-1}(d)$ cannot be the right endpoint of an inappropriate $X$-gap.
Then $d_n\in V_k$, where the set $V_k$ is segment-like and is moved to the left by $f_n$.
If $2^{-n}<d-c$, then we get $d<\min V_k<d_m$ for each $m\geqslant n$, a contradiction.
Hence the proof is finished, since we show that $X_\infty$ is an $L$-Cantorval.
\end{proof}

For a compact set $X\subseteq\bbR$, we adopt the notation from \cite{wal}.
Namely, let $\calA_{X}$ be the family of all $X$-gaps and $\calB_X$ be the family of all $X$-intervals.  Thus $\calA_{X}$ consists of pairwise disjoint open intervals,  $\calB_X$ consists of pairwise disjoint closed intervals and both families are countable. Note that $\bigcup(\calA_X\cup\calB_X)$ is a dense subset of $\bbR$. Indeed, if $x\in \bbR \setminus \bigcup(\calA_X\cup\calB_X)$, then $\{x\}$ is a component of $X$, so $x$ is the limit point of $\bigcup \calA_X$.   
The union $\calA_X\cup\calB_X$ can be naturally equipped with an order, namely $I< J$ whenever  $x< y$ for all $x\in I$ and $y\in J$.

\begin{pro}\label{p1} Let $X$ and $Y$ be compact subsets of the reals. 
There exists an increasing homeomorphism $f\colon \bbR \to \bbR$ such that $f[X] =Y$ if and only if there exists an order isomorphism $F\colon \calA_X\cup\calB_X\rightarrow\calA_Y\cup\calB_Y$ such that $F[\calA_X]=\calA_Y$.
\end{pro}
\begin{proof}If $f\colon \bbR \to \bbR$ is an increasing homeomorphism such that $f[X] =Y$, then  $F(I)= f[I]$ is a required homeomorphism.  

Conversely, assume $F\colon \calA_X \cup \calB_X \to \calA_Y \cup \calB_Y$ is an   isomorphism that carries $\calA_X$ onto $\calA_Y$. If $I \in \calA_X \cup \calB_X$, then let $f_I\colon I \to F(I)$ to be an increasing  homeomorphism.  For each $x\in  I$, put  $f(x) =f_I(x)$. Sets $\bigcup (\calA_X \cup \calB_X)$ and $\bigcup (   \calA_Y \cup \calB_Y)$ are dense in $\bbR$, so  already defined function $f$ is na increasing bijection   between dense subsets of the reals. This bijection can be uniquely extended to a desired homeomorphism.    
 \end{proof}

 If $C\subseteq\bbR$ is a Cantor set, then $\calB_C=\emptyset$ and $\calA_C$ is densely ordered. Any two countable densely ordered sets with the first and the last elements are isomorphic, since they are isomorphic to $\mathbb Q\cap[0,1]$.
 From Proposition \ref{p1} it follows that any compact set $X\subseteq\bbR$ such that $\calB_X=\emptyset$ and $\calA_X$ is densely ordered, is an image of the Cantor set by a homeomorphism of $\bbR$.
In other words, we get a well-known property  of the Cantor set, i.e. if a non-empty compact set $X\subseteq\bbR$ is dense in itself and $\inte_\bbR X=\emptyset$, then $X$ is the image of the Cantor set by an autohomeomorphism of $\bbR$.

Let $\bbQ\cap[0,1]=A_\bbQ\cup B_\bbQ$ be a partition such that $A_\bbQ$ and $B_\bbQ$ are dense in $\bbQ\cap[0,1]$ and $\{0,1\}\subseteq A_\bbQ$.

\begin{lem}\label{l1}
  If $(X,\sqsubset)$ is a countable linear order such that $X=A_X\cup B_X$ and $\min X\in A_X$, and $\max X\in A_X$, and the sets $A_X,B_X$ are disjoint and dense in $X$, then there exists an isomorphism $F$ between $(X,\sqsubset)$ and $(\bbQ\cap[0,1],\leqslant)$ such that $F[A_X]=A_\bbQ$.
\end{lem}

\begin{proof}
Adopt the well-known back-and-forth method.
\end{proof}

The next corollary is a variant of a result by Marta Walczy\'nska, compare \cite[Wniosek 26]{wal}.

\begin{cor}
  If a set $X\subseteq \bbR$ is non-empty, compact and regularly closed such that no  $X$-interval has a free endpoint, then there exists a homeomorphism  $f\colon\bbR\to\bbR$ such that the image $f[X]$ is the model Cantorval.
\end{cor} 
\begin{proof}
Clearly, the interior $\int_\bbR X$  is dense in $X$ and no $X$-gap and $X$-interval have a common endpoint.
Also, between any two $X$-gaps lies an $X$-interval, between an $X$-gap and an $X$-interval lies an $X$-gap, and also between any two $X$-intervals there is an $X$-gap.
Therefore sets $\calA_X$ and $\calB_X$ are dense in $\calA_X\cup\calB_X$, i.e. between any two  elements of $\calA_X\cup\calB_X$ lie an element of $\calA_X$ and an element of $\calB_X$.
By Lemma \ref{l1} there exists an order isomorphism  $F\colon\calA_X\cup\calB_X\rightarrow\calA_C\cup\calB_C$ such that $F[\calA_X]=\calA_C$, where $C$ is the model Cantorval.
By Proposition \ref{p1} there exists a homeomorphism $f$ such that the image $f[X]$ is the model Cantorval.
  \end{proof}

We say that $L$-Cantorval $X\subseteq\bbR$ is \emph{special}, whenever each $X$-interval has the left endpoint free.

\begin{cor}
  Let $X$ be a compact metric space that fulfills condition $(\star)$.
If the union of all non-trivial components of $X$ is dense and each non-trivial component has exactly one free endpoint, then $X$ is homeomorphic to the  special $L$-Cantorval.
\end{cor}

\begin{proof}
  The union of all non-trivial components of $X$ is dense, hence  each clopen subset of $X$ contains a non-trivial component.
  By Theorem \ref{thm11}, $X$ is homeomorphic to a special $L$-Cantorval.
\end{proof}

Suppose that $X \subseteq \bbR$ is a special $L$-Cantorval and $(\calA_X\cup\calB_X, \leqslant)$ is the linear space assigned to $X$. 
We have the following.

\begin{itemize}
	\item[(I).] If $I\in \calA_X$ and $I$ is not the greatest element of $\calA_X\cup\calB_X$, then there exists $J\in \calB_X$ such that $J$ is the immediate successor of $I$. 
	\item[(II).] If $J\in \calB_X $, then there exists $I\in \calA_X$ such that $I$ is the immediate predecessor of $J$.
	\item[(III).] The ordered space $(\calA_X, \leqslant)$ is isomorphic to the set of rational numbers contained in  the interval $[0,1]$ with the natural ordering.
        \end{itemize}
               Indeed, by the definition the right endpoint of  an $X$-gap belongs to  the $X$-interval, which implies (I).
        The condition  of an $L$-Cantorval being special implies (II).
        Suppose $I,J\in\calA_X$ and $I<J$.
        Let $K$ be an $X$-interval containing the right endpoint of $I$.
        Since $\max K<\inf J$, therefore there exists an $X$-gap that lies between $I$ and $J$.
        So, $\calA_X$ is densely ordered, hence isomorphic to $\bbQ\cap[0,1]$.
        
Conditions (III) and (II) imply that   no $I\in \calA_X$ has an immediate predecessor in $(\calA_X\cup\calB_X, \leqslant)$.
 Indeed, let $I\in \calA_X$ and suppose $J\in \calB_X$   is the immediate predecessor of $I$.  According to (II) there exists  $K\in \calA_X$ which is the immediate predecessor of   $J$. Consequently, no element from $\calA_X$ lies between $I$ and $K$, which contradicts (III).

\begin{lem}\label{l2} Let $(A_1\cup B_1, \leqslant)$ and $(A_2\cup B_2, \leqslant)$ be infinite countable   ordered spaces, where $A_i\cap B_i=\emptyset$ for $i\in \{1,2\}$. If the pairs $(A_i,B_i)$ fulfill conditions \textup{(I)--(III)}, where $A_i$  is substituted for $\calA_X$ and $B_i$ for $\calB_X$, then there exists an order  isomorphism $F\colon A_1\cup B_1 \to A_2\cup B_2$ such that $F[A_1] =A_2$. \end{lem}
\begin{proof}
Using  (III), fix an order isomorphism $F:A_1\to A_2$. If $J\in B_1$ and $I\in A_1$ is an immediate predecessor of $J$, them let $F(J)$ be an immediate successor of $F(I)$. Thus, $F\colon A_1\cup B_1 \to A_2\cup B_2$ is a well  defined increasing bijection, because of conditions (I) and (II). 
\end{proof}

\begin{cor}
If $X$ and $Y$ are special $L$-Cantorvals, then there exists an increasing homeomorphism  $f\colon\bbR\to\bbR$ such that  $f[X]=Y$.
\end{cor} 
\begin{proof}
Let $(\calA_X\cup\calB_X,\leqslant)$ and $(\calA_Y\cup\calB_Y,\leqslant)$ be linear orders assigned to $X$ and $Y$, respectively.
Both spaces $(\calA_X\cup\calB_X,\leqslant)$ and $(\calA_Y\cup\calB_Y,\leqslant)$ fulfill conditions (I)--(III).
By Lemma \ref{l2},  there exists an isomorphism $F\colon \calA_X\cup\calB_X\to \calA_Y\cup\calB_Y$ such that $F[\calA_X]=\calA_Y$.
Proposition \ref{p1} gives a desired homeomorphism.
\end{proof}

\section{Non-homeomorphic $L$-Cantorvals}

%Recall a few notions concerning quotient topology.
Let $X$ be a topological space and $\calC_X$ be the family of all components of $X$.
If $V$ is clopen in $X$, then put $V^*=\{C\in\calC_X\colon C\subseteq V\}$.
The family $\{V^*\colon V\mbox{ is a clopen subset of }X\}$ is a base for the topology $\calT_X$ it generates. 
If $f\colon X\to Y$ is a homeomorphism, then $F\colon \calC_X\to \calC_Y$ is given by the formula $F(C)=f[C]$.
Clearly, if $V$ is a clopen subset of $X$, then $F[V^*]=f[V]^*$.
Hence, $F$ is a homeomorphism and if $Z\subseteq\calC_X$, then  $F|_Z$  is a homeomorphism between $Z$ and $F[Z]$ with respect to  inherited topologies.
Moreover, if $f\colon X_0\to X_1$ is a homeomorphism between  $L$-Cantorvals and  $\calZ_i=\{C\in\calC_{X_i}\colon C\mbox{ has no free endpoint}\}$, then subspaces $\calZ_0\subseteq\calC_{X_0}$ and $\calZ_1\subseteq\calC_{X_1}$ are homeomorphic.  

\begin{lem}
Let $\alpha$ be a countable ordinal  and $A\subseteq\alpha$.
There exists an $L$-Cantorval $X$ such that the set $\{C\in\calC_X\colon C\mbox{ has no free endpoint}\}$ with the topology inherited from $\calC_X$ is homeomorphic to $A$ with the topology inherited from linear topology on $\alpha$.
\end{lem}

\begin{proof}
We may assume that $\alpha\subseteq [0,1]$  is equipped  with the topology inherited from $[0,1]$ is compact.
Let $a_0,a_1,\ldots$ be an enumeration of $A$.
The set
$$\textstyle \calI=([0,1]\times\{0\})\cup\bigcup\{\{a_n\}\times[0,\frac1{2^n}]\colon n<\omega\}$$
equipped with the lexicographic ordering is order isomorphic to $[0,1]$, in particular $\calI$ with the order topology is homeomorphic to $[0,1]$.
The set 
$$A^*=([0,1]\setminus\alpha)\times\{0\}$$
is an open subset of $\calI$.
Replacing each component of $A^*$ by a special $L$-Cantorval such that its minimum an maximum are its trivial components, we obtain the desired $L$-Cantorval.
\end{proof}

\begin{thm}
There exists continuum many non-homeomorphic $L$-Cantorvals.
\end{thm}

\begin{proof}
There exists continuum many countable metric spaces which are scattered and non-homeomorphic, see \cite{ms} or compare \cite{pw}.
Each such space is a copy of a subset of some countable ordinal number, see \cite{ku} or compare \cite{pw}, \cite{tel}.
Applying the previous lemma, we are done.
\end{proof}

\end{document}